\documentclass[12pt]{amsart}

\usepackage{amsthm, amsmath, amssymb, verbatim}
\usepackage[pdftex,colorlinks]{hyperref}
\hypersetup{%
    pdftitle={Simplicity of finitely-aligned k-graph C*-algebras},
    pdfauthor={Jacob Shotwell},
    bookmarksnumbered,
    pdfstartview={FitH},
    urlcolor=cyan,
    linkcolor=blue
    }%

\usepackage{amsrefs}

\newcommand{\Obj}{\operatorname{Obj}}
\newcommand{\Hom}{\operatorname{Hom}}
\newcommand{\Cstar}{C^\ast}
\newcommand{\Cstarl}{\Cstar(\Lambda)}
\newcommand{\field}[1]{\mathbb{#1}}
\newcommand{\NN}{\field{N}}

\newcommand{\Lmin}{\Lambda^{\textrm{min}}}
\newcommand{\FE}{\textrm{FE}}

\newtheorem{thm}{Theorem}[section]

\newtheorem{lem}[thm]{Lemma}
\newtheorem{prop}[thm]{Proposition}
\newtheorem{claim}{Claim}

\theoremstyle{definition}
\newtheorem{defn}[thm]{Definition}

\theoremstyle{remark}
\newtheorem{rmk}[thm]{Remark}
\newtheorem{rmks}[thm]{Remarks}

\numberwithin{equation}{section}
\title{Simplicity of finitely-aligned $k$-graph $C^\ast$-algebras}
\author{Jacob Shotwell}
\address{Jacob Shotwell\\ Department of Mathematics and Statistics\\ Arizona State University \\ Tempe, AZ 85281\\ USA} 
\email{shotwell@asu.edu}
\date{\today}

\begin{document}
\maketitle
\begin{abstract}
It is shown that \emph{no local periodicity} is equivalent to the \emph{aperiodicity condition} for arbitrary finitely-aligned $k$-graphs. This allows us to conclude that $\Cstarl$ is simple if and only if $\Lambda$ is cofinal and has no local periodicity. 
\end{abstract}

\section{Introduction}
Kumjian and Pask introduced $k$-graph $\Cstar$-algebras in \cite{kp} as generalizations of the higher-rank Cuntz-Krieger algebras studied by Robertson and Steger in \cite{rs99}. There are two immediate difficulties that arise in the theory of $k$-graphs. The first difficulty is presented by sources. For directed graphs, a source is simply a vertex that receives no edge. For $k$-graphs, a source is a vertex that fails to receive an edge of some degree. The notion of local convexity was introduced in \cite{rsyconvex} in order to associate a $\Cstar$-algebra to certain well-behaved $k$-graphs with sources. The second major obstruction in studying $k$-graphs is presented by infinite receivers. \emph{Finitely-aligned} $k$-graphs were introduced in \cite{finali} in order to associate a $\Cstar$-algebra to row-infinite $k$-graphs graphs (possibly containing sources) that satisfy a mild condition. 

In \cite{kp}, Kumjian and Pask introduce an aperiodicity condition for row-finite $k$-graphs without sources and show that if $\Lambda$ satisfies this aperiodicity condition, then $\Cstarl$ is simple if and only if $\Lambda$ is cofinal. The aperiodicity condition of Kumjian and Pask also serves as a critical hypothesis for a number of important structural results concerning $k$-graph $\Cstar$-algebras. A number of different aperiodicity conditions have appeared in the literature for the variety of classes of $k$-graphs \citelist{\cite{rsyconvex}\cite{finali}\cite{sims20061}\cite{fmy}\cite{robsims20071}\cite{robsims20072}}. 

For row-finite $k$-graph without sources, Robertson and Sims introduce the notion of \emph{no local periodicity} \citelist{\cite{robsims20071}}. This formulation of aperiodicity is formally weaker than the condition introduced by Kumjian and Pask. Nonetheless, Robertson and Sims show that no local periodicity is equivalent to a number of other aperiodicity conditions for row-finite $k$-graphs without sources. The advantage of no local periodicity is that its negation is strong enough to prove that $\Cstarl$ is simple if and only if $\Lambda$ is cofinal and has no local periodicity. Robertson and Sims furthermore use this condition to classify $k$-graph $\Cstar$-algebras in which every ideal is gauge-invariant. This work is similar to the result from directed graph algebras stating that $\Cstar(E)$ is simple if and only E is cofinal and every cycle has an exit.

In \cite{farthing2007}, Farthing constructs a sourceless $k$-graph $\bar{\Lambda}$ from a $k$-graph $\Lambda$ in such a way that $\Cstar({\bar{\Lambda}})$ is Morita equivalent to $\Cstarl$ when $\Lambda$ is row-finite. Robertson and Sims make use of this result in \cite{robsims20071} to generalize their previous work to the locally convex row-finite $k$-graphs. Robertson and Sims' simplicity result is limited to the locally-convex case because of an unexpected difficulty with projecting paths from $\bar{\Lambda}^\infty$ onto $\Lambda^{\leq \infty}$. 

For the finitely-aligned case, a number of aperiodicity conditions have appeared, often defined on different boundary path spaces.  In \cite{finali}, Raeburn, Sims, and Yeend use a similar condition to \emph{Condition B} from \cite{rsyconvex} to prove their version of the Cuntz-Krieger uniqueness theorem. Farthing, Muhly, and Yeend introduce a version of Kumjian and Pasks' aperiodicity condition in \cite{fmy} to prove a version of the Cuntz-Krieger uniqueness theorem using groupoid methods. The condition in \cite{fmy} is much different than that in \cite{rsyconvex}, partly because it operates on the closure of the boundary path space employed by Robertson, Sims, and Yeend. 

In this paper, the work of Robertson and Sims is generalized to the finitely-aligned case. We show that the condition in \cite{fmy} is equivalent to an appropriate formulation of no local periodicity. In Section 2, we briefly introduce the standard definitions and results from the literature. In Section 3, we introduce a condition called \emph{strong no local periodicity} for finitely-aligned $k$-graphs without sources and show that the condition is equivalent to no local periodicity in this situation. This allows us to exactly follow the proof of \cite[Lemma 2.2]{robsims20071} to prove that no local periodicity implies the aperiodicity condition in \cite{fmy}. We then show how to reduce the arbitrary finitely-aligned case to that of no sources by introducing a sourceless $(k-a)$-graph that carries information about aperiodic paths in the original $k$-graph. In Section 4, we use these results to construct the usual simplicity argument as in \cite{robsims20071} and \cite{robsims20072}.

I would like to thank my advisor, Jack Spielberg, for his help with these results.

\section{Preliminaries}

Let $k \in \NN$ and regard $\NN^k$ as a monoid with identity $0$. Let $e_i$ denote the $i^\textrm{th}$ generator of $\NN^k$. For $m, n \in \NN^k$ write $m \leq n$ to mean $m_i \leq n_i$ for $i=1,2,\ldots,k$. For $m,n \in \NN$ let $m \vee n$ and $m \wedge n$ denote the pairwise maximum and minimum of $m$ and $n$, respectively.

\begin{defn} A $k$-graph consists of a countable small category $\Lambda$ together with a functor $d: \Lambda \rightarrow \mathbf{N}^k$ which satisfies the \emph{unique factorization property}: For every $\lambda \in \Lambda$ and $m , n \in \mathbf{N}^k$ such that $d(\lambda)=m+n$, there exist unique $\nu, \mu \in \Lambda$ such that $\lambda = \mu \nu$, $d(\mu)=m$, and $d(\nu)=n$. \end{defn}

Let $\Lambda^n=d^{-1}(n)$ and let $r$ and $s$ denote the range and source maps of $\Lambda$ respectively. $\Obj(\Lambda)$ is naturally identified with $\Lambda^0$ via the unique factorization property and thus $r,s: \Lambda \rightarrow \Lambda^0$. For $v \in \Lambda^0$ and $E \subseteq \Lambda$, put $vE = \{ \mu \in E : r(\mu) = v\}$ and $E v = \{ \mu \in E : s(\mu) = v\}$.

For $n \in \NN^k$, define 
$$\Lambda^{\leq n}=\left\lbrace \lambda \in \Lambda : d(\lambda) \leq n \ \textrm{and} \ d(\lambda) + e_i \leq n \Rightarrow s(\lambda) \Lambda^{e_i} = \emptyset \right\rbrace$$ 

Note that $v \Lambda^{\leq n} \neq \emptyset$ for all $v \in \Lambda^0$ and $n \in \NN^k$. Furthermore, $\Lambda^{\leq n} = \Lambda^n$ if $\Lambda$ has no sources.

 Given $\lambda, \mu \in \Lambda$, a minimal common extension of $\lambda$ and $\mu$ is a pair $(\alpha,\beta) \in \Lambda \times \Lambda$ such that $\lambda \alpha = \mu \beta$ and $d(\lambda \alpha) = d(\lambda) \vee d(\mu)$. The set of minimal common extensions of $\lambda$ and $\mu$ is denoted by $\Lmin(\lambda,\mu)$. Define $\textrm{MCE}(\lambda,\mu)=\{ \lambda \alpha : (\alpha,\beta) \in \Lmin(\lambda,\mu) \}$.

\begin{defn} A $k$-graph $\Lambda$ is \emph{finitely aligned} if $\Lmin(\lambda,\mu)$ is finite for all $\lambda,\mu \in \Lambda$.\end{defn}

\begin{defn} Let $\Lambda$ be a $k$-graph, $v \in \Lambda^0$, and $E \subseteq v \Lambda$. $E$ is \emph{exhaustive} if for every $\mu \in v \Lambda$ there is $\lambda \in E$ such that $\Lmin(\lambda,\mu) \neq \emptyset$. Define 
$$\FE(\Lambda) = \{ F \subseteq v \Lambda\backslash \ | \  v \in \Lambda^0, \ \textrm{F is finite exhaustive}\}.$$\end{defn}

\begin{rmk}
 If $\Lambda$ has no sources, then $\Lambda^{n}$ is exhaustive for all $n \in \NN^k$. More generally, $\Lambda$ is locally convex if and only if $\Lambda^{\leq n}$ is exhaustive for all $n \in \NN^k$.
\end{rmk}

\begin{defn} For $\eta \in \Lambda$ and $F \subseteq r(\eta) \Lambda$, $$\textrm{Ext}(\eta;F) := \bigcup_{\lambda \in F} \{\alpha \in \Lambda \ | \ (\alpha,\beta) \in \Lmin(\eta,\lambda) \ \textrm{for some} \ \beta \in \Lambda\}.$$ 
\end{defn}

If $F \in v \FE(\Lambda)$ and $\eta \in v \Lambda$, then $\textrm{Ext}(\eta;F) \in s(\eta) \FE(\Lambda)$.

\begin{defn}
Let $(\Lambda,d)$ be a finitely aligned $k$-graph. A \emph{Cuntz-Krieger $\Lambda$-family} is a collection of partial isometries $\{ s_\lambda : \lambda \in \Lambda\}$ in a $\Cstar$-algebra $B$ satisfying
\begin{enumerate}
  \item $\{ s_v : v \in \Lambda^0 \}$ is a family of mutually orthogonal projections.
  \item $s_{\lambda} s_{\mu} = s_{\lambda \mu}$ when $s(\lambda) = r(\mu)$.
  \item $s^{\ast}_{\lambda} s_{\mu} = \sum_{(\alpha,\beta) \in \Lmin(\lambda,\mu)} s_\alpha s^\ast_\beta$ for all $\lambda, \mu \in \Lambda$.
  \item $\prod_{\lambda \in E} (s_v - s_\lambda s_\lambda^\ast)=0$ for all $E \in v \FE(\Lambda)$. 
\end{enumerate}
\end{defn}

Denote by $\Cstarl$ the universal $\Cstar$-algebra containing a Cuntz-Krieger $\Lambda$ family.

\subsection{Boundary Paths}
Given a finitely aligned $k$-graph, let $X_\Lambda$ be the collection of graph morphisms $x: \Omega_{k,m} \rightarrow \Lambda$. For such $x$, define $d(x)=m$. Let $\Lambda^{\leq \infty}$ be the collection of paths $x \in X_\Lambda$ for which there is  $n_x \in \NN^k$ such that $n_x \leq d(x)$ and $$n \in \NN^k, n_x \leq n \leq m \ \textrm{and} \ n_i=m_i \ \textrm{imply that} \ x(n)\Lambda^{e_i}=\emptyset.$$ Note that when $\Lambda$ is locally convex, we may take $n_x = 0$. 

Let $\partial \Lambda$ be the collection of paths $x \in X_\Lambda$ such that for all $n \leq d(x)$ and for all finite exhaustive $E \subseteq x(n)\Lambda$, there is $\lambda \in E$ such that $x(n,n+d(\lambda))=\lambda$. We have $\Lambda^{\leq \infty} \subseteq \partial \Lambda$, but $\Lambda^{\leq \infty} \neq \partial \Lambda$ in general. If $\Lambda$ has no sources, then $\Lambda^{\leq \infty} = \Lambda^{\infty}$.

If $x \in \partial \Lambda$ and $n \geq d(x)$, define $\sigma^n x$ by $\sigma^n x (0,p) = x(n,n+p)$ for all $p \leq d(x)-n$. Then $\sigma^n x \in \partial \Lambda$. If $\lambda \in \Lambda x(0)$, there is a unique path $\lambda x \in \partial \Lambda$ such that $\lambda x (0,d(\lambda))=\lambda$ and $\lambda x(0,p) = \lambda x(0,p-d(\lambda))$ for $p \in \NN^k$ satisfying $p+d(\lambda) \leq d(x)$. Recall that, if $x \in \Lambda^{\leq \infty}$, $n \geq d(x)$, and $\lambda \in \Lambda x(0)$, then $\lambda x \in \Lambda^{\leq \infty}$ and $\sigma^n x \in \Lambda^{\leq \infty}$.

For each $\lambda \in \partial \Lambda$, define $S_\lambda \in \mathcal{B}(\ell^2(\partial \Lambda)))$ by 

$$S_\lambda e_x = \begin{cases}
                   e_{\lambda x} &\text{if} \ r(x)=\lambda \\ 
		   0 		&\text{else} \\
                  \end{cases}
$$

\begin{prop} The operators $\{S_\lambda : \lambda \in \Lambda \}$ form a Cuntz-Krieger $\Lambda$-family on $\ell^2(\partial \Lambda)$ such that $S_v \neq 0$. This is called the \emph{boundary-path representation}. \end{prop}

\section{Aperiodicity Conditions}

\begin{defn}Let $\Lambda$ be a finitely-aligned $k$-graph. $\Lambda$ satisfies the \emph{aperiodicity condition} if for every $v \in \Lambda^0$ there exists $x \in v \partial \Lambda$ such that $\sigma^m x = \sigma^n x$ implies $m=n \in \NN^k$ for all $m,n \leq d(x)$. \end{defn}

\begin{defn} Let $\Lambda$ be a finitely-aligned $k$-graph. $\Lambda$ has \emph{no local periodicity} (NLP) if for every $v \in \Lambda^0$ and every $m \neq n \in \NN^k$, there exists $x \in v\partial \Lambda$ such that either $d(x) \ngeq m \vee n$ or $\sigma^m x \neq \sigma^n x$.\end{defn}

\begin{defn} Let $\Lambda$ be a finitely-aligned $k$-graph without sources. $\Lambda$ has \emph{strong no local periodicity} (SNLP) if for every $v \in \Lambda^0$ and every $m \neq n \in \NN^k$, there exists $x \in v \partial \Lambda$ such that $d(x) \geq m \vee n$ and $\sigma^m x \neq \sigma^n x$.
 \end{defn}

\begin{rmks} 
\mbox{}
\begin{itemize}
 \item Robertson and Sims use a different version of no local periodicity for row-finite locally convex $k$-graphs \cite{robsims20072}. For row-finite locally convex $k$-graphs, they prove that the two notions are equivalent.
 \item If no local periodicity fails at $v \in \Lambda^0$, then there are $n \neq m \in \NN^k$ such that $\sigma^n x = \sigma^m x$ for all $x \in v \partial \Lambda$. In this case, $\Lambda$ has \emph{local periodicity $n$, $m$ at $v \in \Lambda^0$}. For row-infinite finitely-aligned $k$-graphs (with or without sources) and fixed $n \neq m \in \NN^k$, there may exist boundary paths $x \in v \partial \Lambda$ such that $d(x) \ngeq n \vee m$. It is not immediately clear whether or not $\Lambda$ can satisfy no local periodicity, yet satisfy $\sigma^n x = \sigma^m x$ whenever $d(x) \geq n \vee m$ for some $n \neq m \in \NN^k$. The next section will establish that this is not possible for finitely-aligned $k$-graphs without sources.
\end{itemize}
\end{rmks}

\subsection{Finitely-aligned, no sources}
Throughout this section, let $\Lambda$ be a finitely-aligned $k$-graph without sources.  First we show that, in this situation, NLP is equivalent to SNLP. This will allow us to use the methods of \cite[Lemma 3.3]{robsims20071} to show equivalence between the aperiodicity condition and no local periodicity. The main strategy is to realize that, if a boundary path has degree with some finite component, then since $\Lambda$ has no sources, we can find infinite receivers along the path. Our strict assumptions in this situation will provide sufficiently many edges to construct an aperiodic boundary path.

\begin{prop} $\Lambda$ satisfies NLP if and only if it satisfies SNLP. \end{prop}
\begin{proof}

It is clear that SNLP implies NLP. Suppose that $\Lambda$ has NLP and fails SNLP at $v \in \Lambda^0$. Then we may fix  $m \neq n \in \NN^k$ such that $\sigma^n y = \sigma^m y$ for all $y \in v \partial \Lambda$ with $d(y) \geq m \vee n$. We will derive a contradiction by constructing $w \in v \partial \Lambda$ satisfying $d(w) \geq n \vee m$ and $\sigma^m w \neq \sigma^n w$. 
Fix $x \in v \Lambda^\infty$. Set $n_1 = n \vee m - m$, $m_1 = n \vee m - n$,  $v_1 = x(n)$, and $v_2 = x(n+n)$. Note that $n_1 \wedge m_1 =0$. 
\begin{claim} $\sigma^{n_1} y = \sigma^{m_1} y$ for each $y \in v_1 \partial \Lambda$ or $y \in v_2 \partial \Lambda$ satisfying $d(y) \geq n_1 \vee m_1$.
\end{claim}

\begin{proof}
 Let $y \in v_1 \partial \Lambda$ satisfy $d(y) \geq n_1 \vee m_1$. Set $w = x(0,n) y$. Then $\sigma^n w = \sigma^m w$, since $d(w) \geq n \vee m$. In particular, $$\sigma^{n \vee m} w = \sigma^{n \vee m - n} \sigma^n w = \sigma^{n \vee m - n} y.$$
Also, $$\sigma^{n \vee m} w = \sigma^{n \vee m - m} \sigma^m w = \sigma^{n \vee m - m} \sigma^n w = \sigma^{n \vee m - m} y.$$
Therefore, $\sigma^{n_1} y = \sigma^{n \vee m - m} y = \sigma^{n \vee m - n} y = \sigma^{m_1} y$, as required.

A similar proof shows that the result holds for each $y \in v_2 \partial \Lambda$.
\end{proof}

\begin{claim} We may assume that either $v_1 \Lambda^{n_1}$ or $v_2 \Lambda^{n_1}$ is finite. \end{claim}
\begin{proof}
 Suppose that both $v_1 \Lambda^{n_1}$ and $v_2 \Lambda^{n_1}$ are infinite sets. Then $v_1 \Lambda^{n}$ and $v_2 \Lambda^{n}$ are also infinite sets because $n \geq n_1$. Also, $x(n+m)\Lambda^{n_1}$ is an infinite set because $x(n+m)=x(n+n)=v_2$. Thus, $$\{x(n,n+m)\alpha \ | \ \alpha \in x(n+m)\Lambda^{n_1} \}$$ is an infinite set. Notice that if $\alpha \in x(n+m) \Lambda^{n_1}$, then
$$d(x(n,n+m) \alpha) = m+n_1=m \vee n.$$ Thus, $x(n,n+m) \alpha \in \textrm{MCE}(x(n,n+m),\lambda)$ for some $\lambda \in v_1 \Lambda^n$. This implies that
$$\bigcup_{\lambda \in v_1 \Lambda^{n}} \textrm{MCE}(x(n,n+m),\lambda)$$ is infinite. Because $\Lambda$ is finitely aligned, $\textrm{MCE}(x(n,n+m),\lambda)$ is finite for each $\lambda \in v_1 \Lambda^{n}$. Hence, $\Lmin(x(n,n+m,\lambda))$ is non-empty for infinitely many $\lambda \in v_1 \Lambda^{n}$. 

By the above work, we may choose $\lambda \in v_1 \Lambda^{n}$ satisfying 
$$\Lmin(x(n,n+m),\lambda) \neq \emptyset$$ and $$\lambda \neq x(m,m+n).$$ Fix $(\alpha,\beta) \in \Lmin(x(n,n+m),\lambda)$, set $\xi = x(0,n+m) \alpha$, and choose $w \in v \Lambda^\infty$ such that $w(0,d(\xi)) = \xi$. Then we have:
\begin{eqnarray*}
\sigma^{n} w(0,n) &=&\lambda \\
\sigma^m w(0,n) &=& x(m,m+n).
\end{eqnarray*}
 Therefore, $\sigma^n w \neq \sigma^m w$. This contradicts our assumption that  $\sigma^n y = \sigma^m y$ for all $y \in v \partial \Lambda$ with $d(y) \geq m \vee n$.
\end{proof}

By Claim 1 and the fact that $\Lambda$ is assumed to satisfy NLP, there is $z \in v_1 \partial \Lambda$ and $i_0 \in \{1,\ldots k\}$ such that $d(z)_{i_0} < (n_1 \vee m_1)_{i_0}$. If $v_1 \Lambda^{n_1}$ is finite, then it is also exhaustive by the assumption of no sources. Hence, the definition of $\partial \Lambda$ gives $\lambda \in v_1 \Lambda^{n_1}$ satisfying $z(0,d(\lambda))=\lambda$. Thus, $d(z) \geq n_1$, which also implies $d(z)_{i_0} < (m_1)_{i_0}$. If $v_1 \Lambda^{n_1}$ is infinite, instead take $z \in v_2 \partial \Lambda$ such that $d(z)_{i_0} < (n_1 \vee m_1)_{i_0}$ for some $i_0 \in \{1,\ldots,k\}$. Since $v_2 \Lambda^{n_1}$ is finite exhaustive, we may similarly conclude that $d(z) \geq n_1$ and $d(z)_{i_0} < (m_1)_{i_0}$. Note that in either case, $n_{i_0}=0$ because $n_1 \wedge m_1 = 0$.

Suppose $v_1 \Lambda^{n_1}$ is finite, let $z \in v_1 \partial \Lambda$ be as above, and set $q=d(z)_{i_0} e_{i_0}$. We claim that $d(z) \geq n_1 + n_1$. To see this, assume otherwise. Fix $\bar{z} \in v_1 \Lambda^{\infty}$ satisfying $\bar{z}(0,q+n_1) = z(0,q+n_1)$. By Claim 1, $\sigma^{n_1} \bar{z} = \sigma^{m_1} \bar{z}$. If $d(z) \ngeq n_1 + n_1$, then $z(n_1 + q)\Lambda^{n_1}$ is infinite (otherwise we could find $\lambda \in z(n_1+q)\Lambda^{n_1}$ such that $z(n_1+q,n_1+q+d(\lambda))=\lambda$, which would give $d(z) \geq n_1 +q + n_1$). Therefore, $\bar{z}(m_1 + q) \Lambda^{n_1}$ is infinite. This is a contradiction of the assumption that $v_1 \Lambda^{n_1}$ is finite. To see this contradiction, recall that \cite[Lemma C.4]{finali} yields that $\textrm{Ext}(\eta;F)$ is finite exhaustive if $F$ is finite exhaustive. In our case, we have assumed that $v_1 \Lambda^{n_1}$ is finite exhaustive, so $\textrm{Ext}(\bar{z}(0,m_1+q),v_1 \Lambda^{n_1})$ is also finite exhaustive since we have assumed $v_1 \Lambda^{n_1} \Lambda \in \FE(\Lambda)$. Moreover, if $\alpha \in \bar{z}(m_1 + q) \Lambda^{n_1}$, then $$d(\bar{z}(0,m_1+q)\alpha) = m_1+q+n_1=(m_1+q) \vee n_1.$$  Therefore, $\alpha \in \textrm{Ext}(\bar{z}(0,m_1+q);v_1 \Lambda^{n_1})$ so that $$\bar{z}(m_1+q)\Lambda^{n_1} \subseteq \textrm{Ext}(\bar{z}(0,m_1+q);v_1 \Lambda^{n_1}).$$  Thus, we can conclude $d(z) \geq n_1 + n_1$. 

Similarly, if $v_2 \Lambda^{n_1}$ is finite, we may take $z \in v_2 \partial \Lambda$ and conclude that $d(z) \geq n_1 + n_1$. Without loss of generality, assume $v_1 \Lambda^{n_1}$ is finite and fix $z, \bar{z}$ as above.

We have $\sigma^{n_1} \bar{z} = \sigma^{m_1} \bar{z}$ by Claim 1, so $\bar{z}(n_1 + q) \Lambda^{e_{i_0}} = \bar{z}(m_1 + q) \Lambda^{e_{i_0}}$ is an infinite set. Also, the above work shows that $\bar{z}(m_1 + q + n_1) \Lambda^{e_{i_0}}$ is infinite. Arguing similarly to the proof of Claim 2,
$$\bigcup_{\lambda \in \bar{z}(m_1+q)\Lambda^{e_{i_0}}} \textrm{MCE}(\bar{z}(m_1+q,m_1+q+n_1),\lambda)$$ is an infinite set.

This implies that $\Lmin(\bar{z}(m_1+q,m_1+q+n_1),\lambda)$ is non-empty for infinitely many $\lambda \in \bar{z}(m_1 + q) \Lambda^{e_{i_0}}$. Therefore, we may choose $\lambda \in \bar{z}(m_1+q)\Lambda^{e_{i_0}}$ such that $$\lambda \neq \bar{z}(n_1+q,n_1+q+e_{i_0})$$ and $$\Lmin(\bar{z}(m_1+q,m_1+q+n_1),\lambda) \neq \emptyset.$$ Let $(\alpha,\beta) \in \Lmin(\bar{z}(m_1+q,m_1+q+n_1),\lambda)$ and set $$\xi = \bar{z}(0,m_1+q+n_1) \alpha.$$ 

Choose $w \in v_1 \Lambda^{\infty}$ such that $w(0,d(\xi)) = \xi$. Then we have
\begin{eqnarray*}
\sigma^{m_1} w (q,q+e_{i_0})  &=& \xi(m_1 + q,m_1 + q + e_{i_0}) = \lambda \\
\sigma^{n_1} w(q,q+e_{i_0}) &=& w(n_1+q,n_1+q+e_{i_0}) \\
 &=& \bar{z}(n_1+q,n_1+q+e_{i_0})
\end{eqnarray*}
 However, $\lambda$ is chosen such that $\lambda \neq \bar{z}(n_1+q,n_1+q+e_{i_0})$. Therefore, $\sigma^{m_1} w \neq \sigma^{n_1} w$, as required.
\end{proof}

\begin{prop} Let $\Lambda$ be a finitely-aligned $k$-graph without sources. The following are equivalent:
 \begin{enumerate}
  \item $\Lambda$ has no local periodicity.
  \item $\Lambda$ satisfies the aperiodicity condition.
 \end{enumerate}
\end{prop}
\begin{proof} It is clear that the aperiodicity condition implies no local periodicity. Assume that $\Lambda$ satisfies NLP and fails the aperiodicity condition. The above work shows we may assume that for every $v \in \Lambda^0$ and $n \neq m \in \NN^k$, there is $x \in v \Lambda^{\infty}$ such that $d(x) \geq n \vee m$ and $\sigma^n x \neq \sigma^m x$. A proof identical to that of \cite[Lemma 3.3]{robsims20071} now shows that $\Lambda$ satisfies the aperiodicity condition.
\end{proof}

\subsection{Finitely-aligned, with sources}
This section is dedicated to proving the following proposition.

\begin{prop} Let $\Lambda$ be a finitely-aligned $k$-graph. Then $\Lambda$ satisfies the aperiodicity condition if and only if $\Lambda$ has no local periodicity.
\end{prop}
\begin{proof}
It is clear that the aperiodicity condition implies no local periodicity. Suppose that $\Lambda$ has no local periodicity but fails the aperiodicity condition at some $v_1 \in \Lambda^0$. 

Assume there exists $x_1 \in v_1 \Lambda^{\leq \infty}$ such that $d(x_1)_{i_1} < \infty$ for some $i_1 \in \{1,\ldots,k\}$. If no such $x_1 \in v_1 \Lambda^{\leq \infty}$ exists, then $v_1 \Lambda^{\leq \infty} = v_1 \Lambda^{\infty}$. Fix $t_1 \in \NN^k$ such that $x_1(t_1)\Lambda^{e_{i_1}} = \emptyset$. Set $v_2 = x_1(t_1)$ and note that $d(y)_{i_1} = 0$ for every $y \in v_2 \Lambda^{\leq \infty}$. 

Suppose there is $x_2 \in v_2 \Lambda^{\leq \infty}$ such that $0 < d(x_2)_{i_2} < \infty$ for some $i_2 \in \{1,\ldots,k\}$. Then $i_1 \neq i_2$ and we may find $t_2 \in \NN^k$ such that $x_2(t_2)\Lambda^{e_{i_2}} = \emptyset$. Set $v_3 = x_2(t_2)$.

We may continue in this fashion to find $v_a \in \Lambda^0$ and an arrangement $\{i_1,\ldots,i_a,i_{a+1},\ldots,i_k\}$ of $\{1,\ldots,k\}$ such that, for every $x \in v_a \Lambda^{\leq \infty}$, 
\[
 d(x)_i = \begin{cases}
  0 & \text{if } i=i_j, \ j \leq a \\
  \infty & \text{if } i=i_j, \ a+1 \leq j \leq k \\
   \end{cases}
\]

Define a category $\Gamma$ by setting:
\begin{eqnarray*}
 \Obj(\Gamma)&=&\{ w \in \Lambda \ | \ v_a \Lambda w \neq \emptyset\} \\ 
  \Hom(\Gamma)&=&\{\lambda \in \Lambda \ | \ v_a \Lambda r(\lambda) \neq \emptyset \}
\end{eqnarray*}
 Define a degree functor $d': \Gamma \rightarrow \NN^{k-a}$ by $$d'(\lambda) = \pi (d(\lambda)),$$  where $\pi : \NN^k \rightarrow \NN^{k-a}$ by $$\pi\left(\sum_{i=1}^k a_i e_i \right) = \sum_{j=a+1}^k a_{i_j} e_{i_j}.$$

\begin{claim} $\Gamma$ is a finitely-aligned $(k-a)$-graph without sources.\end{claim}
\begin{proof}
 It is clear that $\Gamma$ is a category, with range and source maps coming from $\Lambda$. It must be checked that $d'$ is a well-defined functor satisfying unique factorization.

That $d'$ is a well-defined functor follows immediately from its definition. To see that $d'$ satisfies unique factorization, let $\lambda \in \Gamma$ and suppose $d'(\lambda) = m'+n'$, where $m',n' \in \NN^{k-a}$. Set $m = \iota(m')$ and $n=\iota(n')$, where $\iota: \NN^{k-a} \rightarrow \NN^k$ is standard injection. Note that $d(\lambda) = m + n$, since otherwise $d(\lambda)_{i_j} > 0$ for some $j \in \{1,\ldots,a\}$, a contradiction of the fact that $v_a \Lambda^{e_{i_j}} = \emptyset$ for $j \in \{1,\ldots,a\}$. Thus, there are $\mu,\nu \in \Lambda$ such that $\lambda = \mu \nu$, $d(\mu)=m$, and $d(\nu)=n$. It is clear that $d'(\mu)=m'$ and $d'(\nu) = n'$, so $d'$ satisfies unique factorization.

$\Gamma$ is finitely-aligned because $|\Gamma^{\textrm{min}}(\lambda,\mu)| = \infty$ readily implies that $\Lmin(\lambda,\mu)$ is infinite.

Finally, suppose that $w \in \Gamma^0$ is such that $w \Gamma^{e_{i_j}} = \emptyset$ for some $j \in \{a+1,\ldots,k\}$. Fix $\lambda \in v_a \Lambda w \neq \emptyset$ and choose $x \in \Lambda^{\leq \infty}$ such that $x(0,d(\lambda)) = \lambda$. Then $d(x)_{i_j} = 0$, a contradiction. Therefore, $\Gamma$ has no sources
\end{proof}

\begin{claim} $\Gamma$ has NLP.\end{claim}
\begin{proof} Fix $w \in \Gamma^0$ and $m' \neq n' \in \NN^{k-a}$. Let $m = \iota(m'), n=\iota(n')$, where $\iota:\NN^{k-a} \rightarrow \NN^k$ is standard injection. Because $\Lambda$ is assumed to satisfy NLP, there is $x \in w \partial \Lambda$ such that $d(x) \ngeq m \vee n$ or $\sigma^m x \neq \sigma^n x$.

Suppose that $d(x) \ngeq m \vee n$ for some $x \in w \partial \Lambda$. Then $d(x)_i < (m \vee n)_i$ for some $i \in \{1,\ldots,k\}$. Since $(m \vee n)_i = 0$ for $i \in \{i_1,\ldots,i_a\}$, this implies that $d(x)_{i_j} < (m \vee n)_{i_j}$ for some $j \in \{a+1,\ldots,k\}$.  Define $y \in w \partial \Gamma$ by $y(0,l) = x(0,\iota(l))$. Then $d'(y)_{i_j} < (m' \vee n')_{i_j}$, so that $d'(y) \ngeq m' \vee n'$.

Suppose that $\sigma^n x \neq \sigma^m x$ for some $x \in w \partial \Lambda$. Define $y \in w \partial \Gamma$ by $y(0,l) = x(0,\iota(l))$. Note that $d(x)_i = 0$ for $i \in \{i_1,\ldots,i_a\}$. It follows immediately that $\sigma^{n'} y \neq \sigma^{m'} y$.

\end{proof}

\begin{claim} $\Gamma$ fails the aperiodicity condition.\end{claim}
\begin{proof}
 It is assumed that $\Lambda$ fails the aperiodicity condition at $v_1 \in \Lambda^0$. Let $y \in v_a \partial \Gamma$. We will find $n^{\prime} \neq m^{\prime} \in \NN^{k-a}$ such that $\sigma^{n^\prime} y = \sigma^{m^\prime} y$.

For $t \in \NN^k$ define $x \in v_a \partial \Lambda$ by $x(0,t)=y(0,\pi(t))$ and fix $\mu \in v_1 \Lambda v_a$ (using the fact that $v_1 \Lambda v_a \neq \emptyset$ by construction of $v_a$). Since $\Lambda$ fails the aperiodicity condition at $v_1 \in \Lambda^0$, there are $n \neq m \in \NN^k$ such that $\sigma^n (\mu x) = \sigma^m (\mu x)$. Notice that $d(x)_i=0$ when $i \in \{i_1,\ldots,i_a\}$ and that $d(x)_i=\infty$ whenever $m_i \neq n_i$. Thus, $n_i \neq m_i$ for some $i \in \{i_{a+1},\ldots,i_k\}$.

Define $p \in \NN^k$ by 

\[
 p_i = \begin{cases}
  d(\sigma^n(\mu x))_i & \text{if } i=i_j, \ j \leq a \\
  d(\mu)_i & \text{if } i=i_j, \ a+1 \leq j \leq k \\
   \end{cases}
\]

Then $p \leq d(\sigma^n(\mu x))$, $p + n \geq d(\mu)$, and \begin{comment}$$\sigma^{n+p}(\mu x)(0) = (\mu x)(n+p) = x(n+p-d(\mu)).$$\end{comment} 
$$\sigma^{n \vee m - n} \sigma^{n+p}(\mu x) = \sigma^{n \vee m - m} \sigma^{n+p}(\mu x).$$ Let $q = n+p - d(\mu)$. Then $$(n \vee m - n + q) \vee (n \vee m - m + q) \leq d(x)$$ because $((n \vee m - n ) \vee (n \vee m - m ))_i > 0$ implies $d(x)_i = \infty$ and $(n \vee m - n + q) \vee (n \vee m - m + q)_i = q_i \leq d(x)$ otherwise. Moreover,

\begin{eqnarray*}
 \sigma^{n \vee m - n + q} x &=& \sigma^{n \vee m + p} (\mu x) \\
			     &=& \sigma^{n \vee m - n} \sigma^{n+p} (\mu x) \\
	                     &=& \sigma^{n \vee m - m} \sigma^{n+p}(\mu x) \\
                             &=& \sigma^{n \vee m - m + q} x. \\
\end{eqnarray*}

%  Define $y \in x(n+p-d(\mu)) \partial \Gamma$ by $y(0,l) = \sigma^{n+p}(\mu x)(0,\iota(l))$. Then $\sigma^{\pi(m \vee n - n)} y = \sigma^{\pi(m \vee n - m)} y$. Also, if $\pi(n \vee m - n) = \pi(n \vee m - m)$, then $\pi(n) = \pi(m)$, whence $\pi(n)_i = \pi(m)_i$ for each $i \in \{i_{a+1},\ldots,i_k\}$. Hence, $n \neq m$ yields $i \in \{1 \ldots i_a\}$ such that $m_i \neq n_i$. This, however, implies $d(x)_i = \infty$ for some $i \in \{i_1,\ldots,i_a\}$, a contradiction of the fact that $d(x)_i = 0$ for $i \in \{i_1,\ldots,i_a\}$. 

Set $n^{\prime} = \pi(n \vee m - m + q)$ and $m^{\prime}= \pi(n \vee m - n + q)$. Notice that $n^{\prime} \neq m^{\prime}$, since otherwise $(n \vee m - n)_{i_j} = (n \vee m - m)_{i_j}$ for each $j \in \{a+1,\ldots,k\}$. Finally, the above work shows that $\sigma^{n^\prime} y = \sigma^{m^\prime} y$. Therefore, $\Gamma$ fails the aperiodicity condition at $v_a \in \Gamma^0$.
\end{proof}

Claims 4 and 5 yield a contradiction since it is shown above that, for finitely-aligned graphs without sources, NLP is equivalent to the aperiodicity condition. Therefore, NLP implies Condition A.
\end{proof}
\subsection{Equivalent conditions }

The following lemma (and its proof) is identical to \cite[Lemma 3.4]{robsims20071}.
\begin{lem} \label{techlemma}
 Suppose $\Lambda$ has local periodicity $n,m$ at $v \in \Lambda^0$. Then $d(x) \geq n \vee m$ and $\sigma^n x = \sigma^m x$ for every $x \in v \partial \Lambda$. Fix $x \in v \partial \Lambda$ and set $\mu = x(0,m), \alpha = x(m,m \vee n)$, and $\nu=\mu \alpha(0,n)$. Then $\mu \alpha y = \nu \alpha y$ for every $y \in s(\alpha) \partial \Lambda$.
\end{lem}

\begin{proof}
 Let $y \in s(\alpha) \partial \Lambda$ and set $w=\mu \alpha y$. Then we have $d(w) \geq n \vee m$ and $\sigma^n w = \sigma^m w$ by assumption. Moreover, $w(0,n) = \nu$, so $w = \nu \sigma^n w$. Since $\sigma^m w = \sigma^n w$, it follows that $\sigma^n w = \alpha y$, so $\mu \alpha y = w  = \nu \alpha y$.
\end{proof}

\begin{defn} Let $\Lambda$ be a finitely-aligned $k$-graph. $\Lambda$ satisfies Condition B if for each $v \in \Lambda^0$, there is $x \in v \partial \Lambda$ such that $\lambda \neq \mu \in \Lambda v$ implies $\lambda x \neq \mu x$. \cite{fmy}
 \end{defn}

\begin{prop} Let $\Lambda$ be a finitely-aligned $k$-graph. The following are equivalent:
\begin{enumerate}
   \item $\Lambda$ has no local periodicity.
   \item $\Lambda$ satisfies the aperiodicity condition.
   \item $\Lambda$ satisfies Condition B.
\end{enumerate} 
\end{prop}

\begin{proof}
 The above work shows that (1) is equivalent to (2).

(3) $\Rightarrow$ (1). Suppose $\Lambda$ has local periodicity $n,m$ at $v \in \Lambda^0$. Choose $\mu, \nu, \alpha$ as in Lemma \ref{techlemma} and note $d(\mu \alpha) = m \vee n$, $d(\nu \alpha) = n+m \vee n - m$, and that $n+m \vee n -m \neq m \vee n$ if $m \neq n$. Thus, $\mu \alpha \neq \nu \alpha$ and $\mu \alpha y = \nu \alpha y$ for each $y \in s(\alpha) \partial \Lambda$. Therefore, $\Lambda$ fails Condition B at $s(\alpha)$.

(2) $\Rightarrow$ (3). Suppose that $\Lambda$ fails Condition B at $v \in \Lambda^0$. Then for each $x \in v \partial \Lambda$, there are $\lambda_x \neq \mu_x \in \Lambda v$ such that $\lambda_x x = \mu_x x$. Notice $d(\lambda_x) \neq d(\mu_x)$, since then $\lambda_x = (\lambda_x x)(0,d(\lambda_x)) = (\mu_x x)(0,d(\mu_x)) = \mu_x$. 

If $d(\lambda_x)_i \neq d(\mu_x)_i$ for some $i \in \{1,\ldots,k\}$, then $d(x)_i + d(\lambda_x)_i = d(x)_i + d(\mu_x)_i$ implies $d(x)_i = \infty$. Hence, $$(d(\lambda_x) \vee d(\mu_x) - d(\mu_x)) \vee (d(\lambda_x) \vee d(\mu_x) - d(\lambda_x)) \leq d(x).$$ Therefore, 
$$\sigma^{d(\lambda_x) \vee d(\mu_x) - d(\mu_x)} x = \sigma^{d(\lambda_x) \vee d(\mu_x) - d(\mu_x)} \sigma^{d(\mu_x)} (\mu x) = \sigma^{d(\lambda_x) \vee d(\mu_x)} \mu x.$$

Similarly, $$\sigma^{d(\lambda_x) \vee d(\mu_x) - d(\lambda_x)} x = \sigma^{d(\lambda_x) \vee d(\mu_x)} \lambda x.$$

Since we have $\lambda x = \mu x$, this yields $$\sigma^{d(\lambda_x) \vee d(\mu_x) - d(\mu_x)} x = \sigma^{d(\lambda_x) \vee d(\mu_x) - d(\lambda_x)} x.$$ Hence, $\Lambda$ fails the aperiodicity condition at $v \in \Lambda^0$.
\end{proof}

\section{Main Result}
\begin{defn} Let $\Lambda$ be a finitely-aligned $k$-graph and let $H \subseteq \Lambda^0$. $H$ is \emph{hereditary} if, for all $\lambda \in \Lambda$, $r(\lambda) \in H$ implies $s(\lambda) \in H$. $H$ is \emph{saturated} if for all $v \in \Lambda^0$, $F \in v \FE(\Lambda)$ and $s(F) \subseteq H$ imply $v \in H$.
 
\end{defn}

\begin{defn} Let $\Lambda$ be a $k$-graph. $\Lambda$ is \emph{cofinal} if, for every $v \in \Lambda^0$, there is $x \in \partial \Lambda$ and $n \leq d(x)$ such that $v \Lambda x(n) \neq \emptyset$.
 
\end{defn}

\begin{prop} \label{cofinalprop} Let $\Lambda$ be a finitely-aligned $k$-graph. The following are equivalent.
 \begin{enumerate}
  \item $\Lambda$ is cofinal.
  \item If $I$ is an ideal of $\Cstarl$ and $s_v \in I$ for some $v \in \Lambda^0$, then $I=\Cstarl$.
 \end{enumerate}
\end{prop}
\begin{proof}
    (1) $\Rightarrow$ (2). Suppose that $\Lambda$ is cofinal and let $H \subseteq \Lambda^0$ be a non-empty, saturated, and hereditary set. Suppose that $H \neq \Lambda^0$. By \cite[Claim 8.6]{sims20062}, there is a path $x \in \partial \Lambda$ such that $x(n) \notin H$ for all $n \leq d(x)$. This, however, is a contradiction: Let $v \in H$. By the assumption that $\Lambda$ is cofinal, there is $n \leq d(x)$ for which $v \Lambda x(n) \neq \emptyset$. Let $\lambda \in v \Lambda x(n)$. Then $r(\lambda) \in H$ and hence $x(n) = s(\lambda) \in H$ by the assumption that $H$ is hereditary.

    Suppose that $I$ is an ideal of $\Cstarl$ and that $s_v \in I$ for some $v \in \Lambda^0$. Let $H_I = \{v \in \Lambda^0 : s_v \in I\}$. Then $H_I$ is non-empty and \cite[Lemma 3.3]{sims20062} shows that $H_I$ is a saturated and hereditary subset of $\Lambda^0$, whence $H_I = \Lambda^0$. This implies $s_v \in I$ for all $v \in \Lambda^0$, which yields $I = \Cstarl$. 

    (2) $\Rightarrow$ (1). Assume that $\Lambda$ is not cofinal. Then there is a vertex $v \in \Lambda$ and a path $x \in \partial \Lambda$ such that $v \Lambda x(n) = \emptyset$ for all $n \in \NN^k$ with $n \leq d(x)$. Let $$H_x = \{ w \in \Lambda^0 : w \Lambda x(n) = \emptyset \ \forall \ n \in \NN^k \ \textrm{such that} \ n \leq d(x) \}.$$ Then the proof of \cite[Proposition 8.5]{sims20062} shows that $H_x$ is a non-trivial saturated and hereditary set in $\Lambda^0$. Hence, $I_{H_x}$ is a non-trivial ideal of $\Cstarl$ containing a vertex projection.

\end{proof}

\begin{prop} \label{nlpprop} Let $\Lambda$ be a finitely-aligned $k$-graph. The following are equivalent.
 \begin{enumerate}
  \item $\Lambda$ has no local periodicity.
  \item Every non-zero ideal of $\Cstarl$ contains a vertex projection.
  \item The boundary-path representation $\pi_S$ is faithful.
 \end{enumerate}
\end{prop}
\begin{proof}
 (1) $\Rightarrow$ (2). Suppose that $\Lambda$ has no local periodicity. Then $\Lambda$ satisfies the aperiodicity condition. Therefore, the Cuntz-Krieger uniqueness theorem given in \cite[Theorem 7.1]{fmy} yields that every ideal of $\Cstarl$ contains a vertex projection.

(2) $\Rightarrow$ (3). If $\textrm{ker}(\pi_S) \neq \{0\}$, then $s_v \in \textrm{ker}(\pi_S)$ for some $v \in \Lambda^0$, a contradiction.

(3) $\Rightarrow$ (1). Suppose that $\Lambda$ has local periodicity $n,m$ at $v \in \Lambda^0$. Let $\mu,\nu,\alpha$ be as in Lemma \ref{techlemma} and put $a:=s_{\mu \alpha} s^\ast_{\mu \alpha} - s_{\nu \alpha} s^\ast_{\nu \alpha}$. A proof identical to that of \cite[Proposition 3.6]{robsims20071} now shows that $a \in \textrm{ker}(\pi_S) \backslash \{0\}$.
\end{proof}

\begin{thm} Let $\Lambda$ be a finitely-aligned $k$-graph. Then $\Cstarl$ is simple if and only if $\Lambda$ is cofinal and has no local periodicity. \end{thm}

\begin{proof} Propositions \ref{nlpprop} shows that every non-zero ideal of $\Cstarl$ contains a vertex projection. Proposition \ref{cofinalprop} shows that every such ideal is equal to all of $\Cstarl$. Therefore, $\Cstarl$ has no non-trivial ideals.

\end{proof}
 
\bibliography{simplicity}{}

\end{document}